\documentclass[12pt]{amsart}
\usepackage{latexsym,amssymb,amsmath}
\usepackage[usenames]{color}
\usepackage{xcolor}
\usepackage{tikz-cd}
\usepackage{mathtools}
\usepackage{verbatim} % comentarios
\usepackage[all]{xy}
\usepackage[latin1]{inputenc}
\usepackage{mathtools}
\usepackage[colorlinks=true]{hyperref}
\usepackage{thmtools}
\usepackage{leftidx}

\numberwithin{equation}{section}
\newtheorem{theorem}{Theorem}[section]
\newtheorem{lemma}[theorem]{Lemma}
\newtheorem{proposition}[theorem]{Proposition}
\newtheorem{corollary}[theorem]{Corollary}

\newtheorem{theoremx}{Theorem}
\theoremstyle{definition}

\theoremstyle{definition}
\newtheorem{definition}[theorem]{Definition} 
 
\newtheorem{remark}[theorem]{Remark}
\newtheorem{example}[theorem]{Example}

\newcommand{\ZZ}{\mathbb{Z}}

\newcommand{\FF}{\mathbb{F}}

\DeclareMathOperator{\fr}{{fr}}
\DeclareMathOperator{\Hom}{{Hom}}

\DeclareMathOperator{\Ker}{{{Ker}}}

\DeclareMathOperator{\soc}{{{soc}}}

\DeclareMathOperator{\w}{{w}}
\DeclareMathOperator{\supp}{{supp}}

\DeclareMathOperator{\type}{{type}}

\DeclareMathOperator{\Ann}{{Ann}}
\DeclareMathOperator{\dd}{{d}}
\DeclareMathOperator{\hh}{{H}}
\DeclareMathOperator{\hull}{{hull}}
\DeclareMathOperator{\ck}{{Coker}}

\DeclareMathOperator{\rank}{{rank}}

\newcommand{\hd}{\dd_{\hh}}
\newcommand{\n}{\mathfrak{n}}
\newcommand{\m}{\mathfrak{m}}

\newcommand{\cA}{\mathcal{A}}

\newcommand{\cB}{\mathcal{B}}

\begin{document}

%%%%%%%%%%%%%%%%%%%%%%%%%%%%%%%%%%%%%%%%%%%%%%%%%%%%%%%%%%%%%%%%%%%%%
\title{Bounds and MacWilliams Identities for codes over Artinian rings}  

\author[Camps-Moreno]{Eduardo Camps-Moreno}
\address{Eduardo Camps-Moreno\\
Virginia Tech,
Blacksburg, VA, USA}
\email{e.camps@vt.edu}

\author[Espinosa-Vald\'ez]{Carlos Espinosa-Vald\'ez}
\address{Carlos Espinosa-Vald\'ez\\
Centro de Investigaci\'on en Matem\'aticas\\ Guanajuato, Gto., M\'exico.
}
\thanks{C. Espinosa-Vald\'ez was supported by CONAHCyT-SECIHTI Grant 857813 and the SNII Research Assistant Fellowship Exp 61536. H.H. L\'opez was partially supported by the NSF Grant DMS-2401558. L. N\'u\~nez-Betancourt and Y. Pitones were partially supported by SECIHTI Grants CBF 2023-2024-224 and CF-2023-G-33.}
\email{carlos.espinosa@cimat.mx }

\author[L\'opez]{Hiram H. L\'opez}
\address{Hiram H. L\'opez\\
Virginia Tech,
Blacksburg, VA, USA}
%\thanks{}
\email{hhlopez@vt.edu}

\author[N\'u\~nez-Betancourt]{Luis N\'u\~nez-Betancourt}
\address{Luis N\'u\~nez-Betancourt \\ Centro de Investigaci\'on en Matem\'aticas\\ Guanajuato, Gto., M\'exico.}
%\thanks{}
\email{luisnub@cimat.mx}

\author[Pitones]{Yuriko Pitones}
\address{Yuriko Pitones\\ 
Universidad Aut\'onoma Metropolitana, Unidad Iztapalapa.}
%\thanks{$^{\ddd}$ The fifth author was partially supported by SECIHTI Grants CBF 2023-2024-224 and CF-2023-G-33.}
\email{ypitones@xanum.uam.mx }

\maketitle 

\parindent=8mm

\begin{abstract}

This work develops new foundations for the theory of linear codes over local Artinian commutative rings. We use algebraic invariants such as the socle, type, length, and minimal number of generators to measure the size of codes. We prove a relation between the type of a code and the free rank of its dual over Frobenius rings, extending previous results for chain rings.
We also provide new upper bounds for the Hamming distance in terms of the length and type of a code and and conditions under which the dual of an MDS code remains MDS for general Artinian rings. The latter result is obtained by reducing to Frobenius rings via Nagata idealizations, which, to the best of our knowledge, had not been used in coding theory before.
We introduce a conceptually different version of the weight enumerator polynomial. This enumerator is meaningful even in the case of infinite rings and yields new applications in the finite setting. Using this polynomial, we prove a MacWilliams identity that holds over Frobenius rings.

\end{abstract}
%Our approach emphasizes the role of duality and bilinear forms in non-field settings, leading to novel insights into the algebraic and combinatorial behavior of codes over general local Artinian rings.

\setcounter{tocdepth}{1}
\tableofcontents
%%%%%%%%%%%%%%%%%%%%%%%%%%%%%%%%%%
\section{Introduction}\label{SecIntro}
%%%%%%%%%%%%%%%%%%%%%%%%%%%%%%%%%%

A linear code, the main object in coding theory, is typically defined over a finite field. Even though the interest in codes over rings dates back to the early 1970s~\cite{blake72}, their attraction exploded in 1994 when Hammons et al.~\cite{hammons94} proved that the well-known non-linear Kerdock and Preparata codes are projections of linear codes over $\mathbb{Z}_4$. They also proved that these two families of codes are dual to each other, answering the 20-year-old riddle of why they have weight enumerators that satisfy the MacWilliams identities. Other rings that have been considered to define codes are chain rings~\cite{RootCodes, COAAFCCR, LCDCodes} and Galois rings~\cite{bossaller2025, GaloisRings}, which are Artinian rings. However, finite Frobenius rings are undoubtedly the most used rings for defining codes. Two of the main theorems in coding theory are the MacWilliams extension and the MacWilliams identities theorems, which were proven for codes over finite fields by MacWilliams~\cite{FJMacWilliamsThesis}.  Wood~\cite{wood99, wood08} proved that these two theorems hold for codes over finite Frobenius rings. Even more, if any code over a finite ring satisfies any of these theorems, the ring must be Frobenius~\cite{wood08}. Wood's results sparked interest among ring theorists in characterizing other families of rings in terms of the MacWilliams extension theorem~\cite{Iovanov2016,martin19,iovanov22}, showing that coding theory tools inspire ring-theoretic problems. Some families of codes defined over Artinian  rings that have been extensively studied are cyclic codes~\cite{CyclicCodes} and quasi-cyclic codes~\cite{QuasiCyclic1, QuasiCyclic2}. We refer the interested reader to a recent book on codes over finite rings \cite{BookCodesRings}. 

In this manuscript, we aim to further develop the foundations of coding theory focusing on different algebraic invariants that measure size. 
We study linear codes over a local Artinian commutative ring, with unity $(R,\m,K)$, that does not necessarily contain a field. We recall that a commutative ring is Artinian if and only if it is Noetherian and zero-dimensional. 
By an $R$-linear code $C$, we mean an $R$-submodule of $R^n$. The Hamming weight of an element in $C$ is defined as in the classical setting, as the number of its nonzero entries. Since $R$ is an Artinian local ring, there exists a natural bilinear form in $R^n$ that acts as a dot product. Thus, there exists a natural notion of dual code $C^\perp$. In our study, we rely on the following classical concepts of commutative algebra: the socle of $C$, denoted by $\soc(C)$, the type of $C$, defined by $\type(C)=\dim_K \soc(C)$, the minimal number of generators $\mu(C)$, and $\lambda(C)$, the length of $C$.  We have that any of these invariants is zero if and only if $C=0$. Furthermore, the type can only decrease for submodules, the number of generators can only decrease with quotient modules, and the length can only decrease in either case. These features indicate that these invariants give a notion of size, that are meaningful even if the ring is not finite. This conceptual way to measure size allow to recover and extend results previously know only for finite rings. Furthermore, using this idea, we are able to obtain results that are new even for finite rings.

We now focus on the setting of Frobenius rings. We point out that commutative Frobenius  rings are Gorenstein Artinian rings, which are ubiquitous in commutative algebra~\cite{Bass}. We note that $R$ is a Frobenius ring if and only if
$$\lambda(C)+\lambda(C^{\perp})=n \lambda(R)$$
for every $C\subseteq R^{n}$. Samei and Mahmoudi~\cite{Samei_Mahmoud} showed that if $R$ is a chain ring, then $\mu(C)+\fr(C^\perp)=n$, where $\fr(C^\perp)$ is the free rank of $C^\perp$. Since in chain rings the minimum number of generators and the type of a code coincide, the following theorem extends Samei and Mahmoudi's result \cite{Samei_Mahmoud} to codes over Artinian local rings. Furthermore, we prove that the equation also characterizes Frobenius rings. 
\begin{theoremx}[{Theorem~\ref{ThmTypeFreeRank}}] \label{ThmB}
A ring $R$ is Frobenius if and only if 
$$
\type(C)+\fr(C^\perp)=n
$$
for every code $R$-code $C\subseteq R^{n}$.
\end{theoremx}
It is worth mentioning that our proofs are conceptually different from those provided for chain rings~\cite{Samei_Mahmoud}, as we heavily rely on the fact that $R^n$ has a nice bilinear product and enjoys duality and their proof uses the structure of finitely generated modules over chain rings.

 Unlike the classical inner product over the real numbers, $C \cap C^\perp$ could be nonzero. This gives rise to the hull of a code: ${\hull}(C)=C \cap C^\perp$. The hull helps to evaluate the complexity of certain coding theory algorithms, such as verifying the equivalence between two linear codes over finite fields~\cite{Leon1982, Sendrier2000}. If ${\hull}(C)=0$, the code $C$ is called LCD. We prove the following result, which extends previous work by  Bhowmick, Fotue-Tabue, Mart\'inez-Moro, Bandi,
and Bagchi \cite{LCD}. Specifically, in Theorem~\ref{ThmLCDfree}, we show that 
if $R$ is Frobenius and $C$ is LCD, then $C$ is free.

In Section~\ref{SecCodes}, we study the general case when $(R,\m, K)$ is a local Artinian ring. Samei and Mahmoudi~\cite{Samei_Mahmoud} showed that if $R$ is a chain ring, then $\hd(C)\leq n-\mu(C)+1$. 
We note that in a chain ring $\type(C)=\mu(C)$.
We extend this result to all Artinian local rings.
\begin{theoremx}[{Theorem~\ref{ThmSing}}] \label{ThmA}
If $R$ is a local Artinian ring, we have that
\[
\hd(C)\leq n-\frac{\lambda(C)}{\lambda(R)}+1
\qquad
\text{ and }
\qquad
\hd(C)\leq n-\frac{\type(C)}{\type(R)}+1.
\]
\end{theoremx}
The inequalities presented in the previous Theorem are called MDS, for Maximum Distance Separable, and MDT bounds, for Maximum Distance Type, respectively. A code achieving the MDS (MDT) bound is called an MDS (MDT) code. Given that the free rank and the type of a code coincide if and only if the code is free. It turns out that MDS codes are free MDT codes (Theorem~\ref{ThmEquivMDS}). Example~\ref{ex.mdtnotmdr} shows that non-free MDT codes exist.

It is natural to ask whether the dual of an MDS code is also MDS. This is a well know fact for linear codes over a field. 
We show that this property holds for codes over a local Artinian ring in general, which was unexpected for the authors.

\begin{theoremx}[{Theorem~\ref{Theo.MDSduals}}]
If $R$ is a local Artinian ring, then $C$ is MDS if and only if $C^\perp$ is MDS.
\end{theoremx}

The proof of this theorem relies on lifting the code to $S^n$, where $S$ is the Nagata idealization of the canonical module of $R$. By this construction, $S$ is a Frobenius ring that contains $R$. This way we reduce the problem to Frobenius rings where we have a well behaved duality. We expect that more results for general Artinian rings can be obtained following this strategy.

Finally, in Section~\ref{SecMacWilliams}, we explore the notion of a weight enumerator of a code. Since we cannot use cardinality because the ring may be infinite, we  use the length  of the shortened codes of $C$ to address the size of the code. 
In fact, we introduce  $i$-th weight polynomial $g^C_{i}(z)$, with an auxiliary variable $z$. If the ring is finite with a residue field of size $q$, then  $g^C_{i}(q)$ gives the number of codewords of weight $i$.   
The weight enumerator of $C$ is defined by
$$W_C(x,y,z)=\sum^n_{i=1} g^C_{i}(z)x^{n-i}y^{i}.$$
This extends the classical weight enumerator in the finite case. Furthermore, even over finite fields, this polynomial provides new information. For instance, $g^C_{i}(z)$ captures the change in the number of codewords of weight $i$ under field extensions (see Corollary \ref{CorExtension} and Example \ref{ExFieldExt}). With these polynomials on hand, we are able to prove a version of the
MacWilliams Identity that works even for infinite rings.

\begin{theoremx}[{Theorem~\ref{Theo.MWid}}] {\rm (MacWilliams Identity)} If $R$ is Frobenius, then $$W_{C^\perp}(x,y,z)=\frac{1}{z^{\lambda(C)}}W_C(x+zy-y,x-y,z).$$
\end{theoremx}

Furthermore, in Corollary~\ref{25.06.13}, we prove that if the previous identities hold for any code over a fixed ring, the ring must be Frobenius.

In this paper, we consider only commutative rings with unity. $(R,\m,K)$ denotes a local Artinian ring with maximal ideal $\m$ and residue field $K=R/\m$. We also use the terms `codes' and `linear codes' interchangeably because we deal only with codes that are $R$-modules.
%\end{notation}

%%%%%%%%%%%%%%%%%%%%%%%%%%%%%%%%%%
\section{Preliminaries on commutative rings}\label{SecPre}

In this section, we introduce basic and well-known concepts and results in commutative algebra, for proofs and further details, see for example \cite{BrunsHerzog,Eisenbud,Matsumura}.

For the rest of this section, let $C$ be a finitely generated (f.g.) $R$-module.\\
$\bullet$ The socle of $C$ is denoted by $\soc(C)$. Recall $\soc(C)=\Ann_C (\m) \simeq \Hom_R(R/\m, C)$.\\
$\bullet$ The type of $C$ is defined by $\type(C)=\dim_K \soc(C)$.\\
$\bullet$ The minimal number of generators of $C$ is denoted by $\mu(C)$. Note $\mu(C)=\dim_K (C/\m C)$. Lemma~\ref{lemmaDualitySoc} shows that $\mu$ and $\type$ are dual notions.\\
$\bullet$ The length of $C$ is given by
\[
\lambda(C) = \sup \{  n \mid \exists \  0 = C_0 \subseteq C_1 \subseteq \cdots \subseteq C_ n = C, \text{ }C_ i \not= C_{i + 1} \}.
\]
For any maximal chain $0 = C_0 \subsetneq C_1 \subsetneq C_2 \subsetneq \cdots \subsetneq C_t = C$ of submodules, we have $t=\lambda(C)$ and $C_{i}/C_{i-1}$ is simple. In addition, as $R$ is Artinian and $C$ is f.g., then $\lambda(C) < \infty$. The following conditions are equivalent. (1) $C$ is simple, (2) $\lambda(C)=1$, and (3) $C\cong R/\m$. 

%As $C$ is f.g., then $\m^{t}C=0$ for some $t$.

\begin{remark}\label{RemLenCardinality}
If $R=K$, then $\lambda(C)=\dim_K (C)$. When $R$ is finite, $|C|=|K|^{\lambda(C)}$.
\end{remark}
\begin{lemma}\label{aditividad}
Let $0 \to C' \to C \to C'' \to 0$ be a short exact sequence of f.g. $R$-modules. The following holds.
\begin{enumerate}
\item $\mu(C)\geq \mu(C'')$.
\item $\type(C)\geq \type(C')$.
\item $\lambda(C)=\lambda(C')+\lambda(C'')$.
\end{enumerate}
\end{lemma}
Let $E_R(K)$ be the injective hull of $K=R/\m$ as an $R$-module. When the context is clear, we write $E_R$, or even $E$, instead of $E_R(K)$.
\begin{theorem}
We have
$\lambda(E_{R}) < \infty$ and $\lambda(E_{R}) = \lambda(R)$.
\end{theorem}
\begin{definition}
The Matlis dual of $C$ is denoted by $(-)^{\vee}=\Hom_{R}(-,E_R)$.
\end{definition}
As $R$ is Artinian, the functor $(-)^\vee$ gives a category equivalence between f.g. $R$-modules to itself.

\begin{lemma}\label{lengthMatlis}
We have $\lambda(C)=\lambda(C^\vee)$.
\end{lemma}

We now review Frobenius rings and list some of their main properties. For details about Frobenius rings, the interested reader is referred to the books by Bruns and Herzog~\cite{BrHe} and Lam~\cite{Lam2012-jr}, as well as the paper by Bass~\cite{Bass}.

In general, an arbitrary ring $Q$ is called quasi-Frobenius (QF) if $Q$ is Noetherian and self-injective ($Q$ is injective as a $Q$-module). Then, $Q$ is called Frobenius if $Q$ is QF and $\soc(Q_Q) \cong \overline{Q}_Q$, where $\overline{Q}=Q/rad(Q)$. By Lam~\cite[Theorem 16.14]{Lam2012-jr}, $Q$ is Frobenius if and only if $\soc(Q_Q) \cong \overline{Q}_Q$ and $\soc(\leftidx{_Q}Q) \cong \leftidx{_Q}~\overline{Q}$. In our case, $R$ is a commutative local ring, so we adopt the following definition.
\begin{definition}
A ring $R$ is Frobenius if $\soc(R) \cong K$.
\end{definition}
In the context of commutative algebra, $R$ is a Gorenstein ring when it is self-injective. This is equivalent to saying that $\soc(R) \cong K$. In other words, Frobenius rings are the same as Gorenstein rings.

\begin{lemma}\label{Lemma2.7}
$R$ is Frobenius if and only if $\Hom_R(-,R)$ is an exact functor. Equivalently, $R\cong E_R$.
\end{lemma}
The following result helps to construct Frobenius rings.
\begin{proposition}
Let $(S, \n, L)$ be a Noetherian regular local ring of dimension $d$. If $x_{1},\ldots,x_{d}$ is a regular sequence in $S$, then $S/(x_{1},\ldots,x_{t})S$ is Frobenius ring.
\end{proposition}

\begin{example}
The following rings are examples of zero-dimensional Frobenius local rings.
\begin{enumerate}
\item A field.
\item An Artinian chain ring.
\item $\frac{K[x_1,\ldots,x_t]}{(x^{\alpha_1}_1,\ldots,x^{\alpha_t}_t)}$, where $K$ is any field.
\item $\frac{\ZZ}{(p^e)}$.
\end{enumerate}
\end{example}

It is worth noting that non-local Frobenius rings are also relevant in coding theory. Some examples are quasi-cyclic codes over $R = K[x] / \langle x^{m} - 1 \rangle$, where $K$ is a finite field~\cite{QuasiCyclic1,QuasiCyclic2}, cyclic codes over Galois rings~\cite{GaloisRings}, and $\frac{K[x_1,\ldots,x_t]}{(f_1(x_1),\ldots, f_t(x_t))}$~\cite{COAAFCCR}, where $f_i(x_i) \in K[x_i]$ is a monic polynomial.

%%%%%%%%%%%%%%%%%%%%%%%%%%%%%%%%%%%%%%%%%%%%%%%%%%%%%%%%%%%%%%%%%%%%
\section{Codes over Frobenius rings (not necessarily finite)}\label{Sect Frob}
In this section, $(R,\m,K)$ denotes an Artinian Frobenius local ring and $C$ is an $R$-code of length $n$ (in the coding-theoretic sense), meaning that $C$ is an $R$-submodule of $R^n$. Note that if $R$ is Artinian, then it is Noetherian and has Krull dimension zero.%. Note that if $R$ is Artinian, then $R$ is Noetherian and has dimension zero.

\begin{remark}
We point out that a Frobenius ring is a finite product of Frobenius local rings. So, the results of this section can be generalized to Artinian rings that are Frobenius.    
\end{remark}

The minimum distance of $C$ (when $C\neq 0$) is defined by 
\[
\hd(C)=\min\{\w_H(v)\;|\; v\in C\setminus{0}\},
\]
where \(\w_H(v)=|\{i\;|\; v_{i}\neq 0\}|\)
denotes the Hamming weight of $v$.
\begin{definition}
A bilinear product on $R^n$ is given by
$$
v\cdot w=v_1w_1+\ldots+v_nw_n,
$$
where $v=(v_1,\ldots,v_n)$ and $w=(w_1,\ldots,w_n)$ are elements in $R^n$. The orthogonal code of $C$ over $R$ is defined by
\[
C^{\perp_R}=\{v\in R^{n} \mid v\cdot w=0 \text{ for all }w\in C\}.
\]
When the context is clear, we use $C^{\perp}$ instead of $C^{\perp_R}$.
\end{definition}

We start by giving some properties that an $R$-code shares with its socle. The following result shows that the minimum distance of a code depends only on the socle of the code.
\begin{proposition}\label{LemmaDist-DistSoocle}
We have that 
\(\hd(C)=\hd(\soc(C))\).
\end{proposition}
\begin{proof}
Since $\soc(C)\subseteq C$, we have that
\(\hd(\soc(C))\geq \hd(C).\)
Let $v\in C$ such that $w_H(v)=\hd(C)$. We have $(Rv)\cap\soc(C)\neq 0$. Thus, there exists $r\in R$
such that $rv\in\soc(C)$. Consequently,
\[
\hd(C)=\w_H(v)\geq \w_H(rv)\geq \hd(\soc(C)),
\]
which completes the proof.
\end{proof}

%We point out that Lemma~\ref{Lemma_Dimension_analogous} is probably already known to experts in commutative algebra.

% \camps{Es necesario el siguiente Remark?}

% \begin{remark}\label{RemPerpIntersecSum}
% Suppose that $R$ is a Frobenius ring.
% We note that for two codes $C_1$ and $C_2$ in $R^n$, we have that 
% \begin{enumerate}
% \item $(C_1\cap C_2)^\perp=C^\perp_1+C^\perp_2$;
% \item $(C_1+C_2)^\perp=C^\perp_1\cap C^\perp_2$.
% \end{enumerate}
% \end{remark}

\begin{lemma}\label{lemmaDualitySoc}
We have that $\soc(C)^{\vee} \cong C^{\vee}/\m C^{\vee}$ and $\type(C) = \mu(C^\vee)$.
\end{lemma}
\begin{proof}
Let $f_{1},\ldots,f_{t}\in\m$ be generators of $\m$. Consider the exact sequence
\[
\xymatrix{
0\ar[r] & \soc(C)\ar[r] & C\ar[r]^-{\varphi} & C^{t},
}
\]
where $\varphi(w)=(f_{1}w,\ldots,f_{t}w)$. After applying Matlis duality, we obtain the exact sequence
\[
\xymatrix{
0 & \soc(C)^{\vee}\ar[l] & C^{\vee}\ar[l] & (C^{\vee})^{t}\ar[l]_{\varphi^{\vee}},
}
\]
where $\varphi^{\vee}(z_{1},\ldots,z_{t})=f_{1}z_{1}+\cdots+f_{t}z_{t}$. Since $\ck(\varphi^{\vee})=C^{\vee}/\m C^{\vee}$,
it follows that $C^{\vee}/\m C^{\vee}\cong\soc(C)^{\vee}$.
Then,
$$\mu(C^\vee)=\lambda(C^\vee/\m C^\vee)=\lambda(\soc(C)^\vee)=\lambda(\soc(C))=\type(C),$$
which completes the proof.
\end{proof}

Let $\rho:R^n \to K^n$ be the projection given by identifying $K\cong R/\m$. Note that the image $\rho(C)$ of an $R$-code $C\subseteq R^n$ is a $K$-linear subspace of $K^n$. Fix a generator $z$ for $\soc(R)$. The identification $K\cong Rz=\Ann_R \m$ naturally extends to $\alpha:K^n\to \soc(R^n)$. Since $\soc(C)\subseteq \soc(R^n)$, we identify $\soc(C)$ with $\alpha^{-1}(\soc(C))$, which is a $K$-linear subspace of $K^n$.
\begin{lemma}\label{LemmaTypeFreeRank}
We have that
\[
\alpha(\soc(C))^{\perp_K}=\rho(C^{\perp_R}).
\]
That is, the $K$-orthogonal of $\soc(C)$ coincides with the $\rho$-image of the $R$-orthogonal of $C$.
\end{lemma}

\begin{proof}
($\subseteq$) Take $u=([u_{1}],\ldots,[u_{n}])\in\alpha(\soc(C))\subseteq K^n$ and $v=(v_{1},\ldots,v_{n})\in C^{\perp_R}$. Then, $(u_{1}z,\ldots,u_{n}z)\cdot v=(\sum_{i=1}^{n}u_{i}v_{i})z=0$ because $(u_{1}z,\ldots,u_{n}z)\in C$. Note that $\rho(v)=([v_1],\ldots,[v_n])$. We have that $u\cdot v=\sum_{i=1}^{n}[u_{i}][v_{i}]=[\sum_{i=1}^{n}u_{i}v_{i}]$. Since $0=z\cdot (\sum_{i=1}^{n}u_{i}v_{i})$, we obtain that $\sum_{i=1}^{n}u_{i}v_{i}\in \Ann_R z\subseteq \m $. Thus, $\sum_{i=1}^{n}[u_{i}v_{i}]=0$. We conclude that $u\in (\rho(C^{\perp_R}))^{\perp_K}$.

($\supseteq$) Take $a=(a_{1},\ldots,a_{n})\in C^{\perp}$ and $\rho(b)=([b_{1}],\ldots,[b_{n}])\in (\rho(C^{\perp_R}))^{\perp_K}$. Since $[0]=\rho(a)\cdot\rho(b)=\sum_{i=1}^{n}[a_{i}][b_{i}]=\sum_{i=1}^{n}[a_{i}b_{i}]$, we have $\sum_{i=1}^{n}a_{i}b_{i}\in \m$. Observe that $a\cdot\alpha^{-1}(\rho( b))=\sum_{i=1}^{n}a_{i}b_{i}z=0$ because $\m z=0$ by our choice of $z$. 
%Then, $w\in C^{\perp_R}$. 
This implies that $\alpha^{-1}(\rho( b))\in (C^{\perp_R})^{\perp_R}=C$ and $\alpha^{-1}(\rho( b))\in \soc(R^n)$. Therefore, $ \alpha^{-1}(\rho( b))\in \soc(C)$, and so,  $\rho( b)\in \alpha(C)$. This concludes the proof.
\end{proof}

We now prove that the property of $R$ of being Frobenius can be translated to the properties of $R$-codes. Some of these results were known previously only for finite rings.

In the finite ring setting, codes over Frobenius rings have been amply studied since these codes satisfy both the MacWilliams extension theorem and the property that $|C|\cdot |C^\perp|=|R^n|$ for any $C\subseteq R^n$. The following lemma provides an extension of this result for codes over infinite rings. The proof relies on the Matlis Duality. By Remark~\ref{RemLenCardinality}, we recover the classical case over finite Frobenius rings.
\begin{lemma}\label{Lemma_Dimension_analogous} 
We have that $R$ is a Frobenius ring if and only if
$$\lambda(C)+\lambda(C^{\perp})=n \lambda(R)$$
for every code  $C\subseteq R^{n}$.
\end{lemma}
\begin{proof}
We first assume that $R$ is a Frobenius ring.
Consider the exact sequence
\[
\xymatrix{
0\ar[r] & C^{\perp} \ar[r] & R^{n}\ar[r]  & R^{n}/C^{\perp}\ar[r] & 0.
}
\]
By Lemma~\ref{aditividad}~(3), we have that 
$\lambda(R^{n})=\lambda(C^{\perp})+\lambda(R^{n}/C^{\perp})$. It is enough to prove that $\lambda(R^{n}/C^{\perp})=\lambda(C)$. First, we show that $R^{n}/C^{\perp}\cong \Hom_{R}(C,E)$. 
Let $\phi:R^{n}\rightarrow \Hom_{R}(C,R)$, given by $x\mapsto \phi_{x}(y)=x\cdot y$. By Lemma \ref{Lemma2.7}, since $R$ is an injective $R$-module, the functor $\Hom_{R}(-,R)$ is exact. Hence, for the inclusion $C \hookrightarrow R^{n}$, the induced map 
$\Hom_R(R^{n},R) \to \Hom_R(C,R)$ is surjective. This shows that every homomorphism from $C$ to $R$ can be realized as $\phi_x$ for some $x \in R^{n}$, 
and therefore $\phi$ is surjective.
%By the injection $C\hookrightarrow R^{n}$, we have that $\Hom_{R}(R^{n},R)\rightarrow \Hom_{R}(C,R)$ is surjective. This implies that $\phi:R^{n}\rightarrow \Hom_{R}(C,R)$ is surjective. 
In addition, $\ker \phi=\{x\in R^{n}\mid \phi_{x}(y)=0 \text{ for all } y\in C \}=C^{\perp}$. 
Therefore, $R^{n}/C^{\perp}\cong \Hom_{R}(C,R)$, and by Lemma~\ref{lengthMatlis}, $\lambda(R^{n}/C^{\perp})=\lambda(C^{\vee})=\lambda(C)$.

We now assume that $\lambda(C)+\lambda(C^{\perp})=n \lambda(R)$ for every $R$-code $C\subseteq R^n$. In particular, 
$$\lambda(\m)+\lambda(\m^{\perp})= \lambda(R).$$
Since $\m^{\perp}=\soc(R)$, we have that
$$
\type(R)=\dim_K(\soc(R))=\lambda(\soc(R))= \lambda(R) - \lambda(\m)=1.
$$
Hence, $R$ is Frobenius.
\end{proof}

We now introduce the concept of free code and free rank of a code. We then present a bound on the free rank of an $R$-code and give characterizations to determine if a code is free.
\begin{definition}
We say that $C$ is a free $R$-code if it is a free $R$-module. The free rank of $C$, detonated by $\fr(C)$, is defined as the maximum rank of a free summand of $C$. That is, $C\cong R^{\fr(C)}\oplus N$, where $N$ has no free direct summand. 
\end{definition}

Since $R$ is local, every projective finitely generated $R$-module is free. Then,
$$
\fr(C)=\max\{\theta\; |\; \exists  R-\hbox{linear surjection } C\twoheadrightarrow R^\theta\},
$$

\begin{theorem}\label{ThmTypeFreeRank}
$R$ is a Frobenius ring if and only if 
$$
\type(C)+\fr(C^\perp)=n
$$
for every code $R$-code $C\subseteq R^{n}$.
\end{theorem}
\begin{proof}
We first assume that $R$ is a Frobenius ring.
We have that
\[
\dim_{K}(\rho(C^{\perp_R})^{\perp_K})=n-\dim_{K}(\rho(C^{\perp})).
\]
Observe that $\fr(C^{\perp_R})=\dim_{K}(\rho(C^{\perp_R}))$.
We also have 
\[\dim_{K}(\soc(C))=\dim_{K}(\alpha(\rho(C^\perp)^{\perp}))
= n-\dim_{K}(\rho(C^{\perp}))
= n-\fr(C^{\perp})\]
because $\dim_{K}(\soc(C))=\dim_{K}(\rho(C^{\perp_R})^{\perp_K})$ by Lemma~\ref{LemmaTypeFreeRank}.

We now assume that
$$\type(C)+\fr(C^\perp)=n$$
for every code $C\subseteq R^{n}$. If we take $C=\soc(\m)$, then $C^\perp=\m$ and $\fr(C^\perp) = 0$. Thus,
$$\type(R)=\type(\soc(\m))=1-\fr(C^\perp)=1.$$
Hence, $R$ is a Frobenius ring.
\end{proof}

As a consequence of Theorem~\ref{ThmTypeFreeRank}, we recover a result that relates the number of generators and the free rank of a code over a chain ring.
\begin{corollary}[{\cite[Theorem 3.7(1)]{Samei_Mahmoud}}]
Let $R$ be a chain ring and $C\subseteq R^{n}$ an $R$-code. Then,
$$
\mu(C)+\fr(C^\perp)=n.
$$
\end{corollary}
\begin{proof}
In this case, $\m$ is principal, and we can pick a generator $x$.
Then, for every finitely generated module $C$, there exists a sequence of positive integers $a_1,\ldots,a_k$ such that
$$
C\cong \frac{R}{(x^{a_1})}\oplus\cdots\oplus \frac{R}{(x^{a_k})}.
$$
Thus, $\mu(C)=k=\type(R).$ Then, the result follows from Theorem~\ref{ThmTypeFreeRank}.
\end{proof}

We can relate the free rank to the ability of a code to represent information independently and without undue redundancy. As $R$ is Frobenius, and therefore injective, we have an analogue characterization for injections in terms of the projection $\pi$.

\begin{lemma}\label{freerank}
We have that
$$\fr(C)=\dim_K \pi(C)=\max\{\theta\; |\; \exists R-\hbox{linear injection } R^\theta\hookrightarrow C\},$$
where $\rho$ is the projection $\rho:R^n \to K^n$ given by identifying $K\cong R/\m$.
\end{lemma}
\begin{proof}
Let $\ell=\fr(C)$, $r=\dim_K \pi(C),$
and $s=\max\{\theta\; |\; \exists R-\hbox{linear injection } R^\theta\hookrightarrow C\}.$
By definition, we note that $\ell\leq s$.
Fix an injection $\phi:R^s\hookrightarrow C$.
Then, we have that $C=R^s\oplus N$ for some $R$-submodule $N\subseteq C$.
Let $v_i=\phi(e_i)$, where $\{e_1,\ldots, e_s\}$ is the canonical basis of $R^s$.  Since $R^s$ is an injective $R$-module, we have that  $\{v_1,\ldots, v_s\}$ is part of a basis for $R^n$. This means that $\rho(v_1),\ldots,\rho(v_s)$ are linearly independent in $K^n$. Hence, $s\leq r$.

Let $u_1,\ldots,u_r\in C$ be such that $\pi(u_1),\ldots,\pi(u_r)$ are linearly independent in $K^n=R^n/\m R^n$. Then,  $u_1,\ldots,u_r$ form part of a basis of $R^n$. Choose $a_1,\ldots,a_{n-r}\in R^n$ such that $u_1,\ldots,u_r, a_1,\ldots,a_{n-r}$ is a basis for $R^n$. Now, there exists a map $\varphi:R^n\to R^r$ defined by $u_i\mapsto e_i$ and $a_i\mapsto 0$. Composing the inclusion with  $\varphi$, we obtain a surjection $C\to R^r$, and so, $r\leq \ell$.
By combining the inequalities, we get $r=\ell=s.$
\end{proof}

\begin{theorem}\label{Lemma_Free_iff_Free}
If $R$ is Frobenius, $C$ is a free $R$-code if and only if $C^{\perp}$ is a free $R$-code.
\end{theorem}
\begin{proof}
Suppose that $C$ is free. Since $R$ is a Frobenius ring, it is an injective $R$-module over itself. Then, $R^n$ is isomorphic as an $R$-module to  $C\oplus R^n/C$.
Hence, $R^n/C$ is a free $R$-module. 
Let  $v_1,\ldots,v_\ell$ be a basis for $C$ and  let $A$ be the matrix whose columns are $v_1,\ldots,v_\ell$. We have $\Ker(A^T)=C^\perp$ and there is an exact sequence: 
\[
\xymatrix{
0\ar[r] & C^\perp\ar[r] & R^n\ar[r]^-{A^{T}} & R^\ell.
}
\]

After applying Matlis duality, we obtain another exact sequence: 
\[
\xymatrix{
R^\ell\ar[r]^-{A} & R^n\ar[r] & (C^\perp)^\vee\ar[r] & 0.
}
\]
Hence, $(C^\perp)^\vee\cong R^n/C$, and so $(C^\perp)^\vee$ is free.
Since $C^\perp=((C^\perp)^\vee)^\vee$, and the dual of a free module is free, we conclude that $C^\perp$ is free. The converse follows from the fact that $(C^\perp)^\perp=C$.
\end{proof}

We now characterize free $R$-codes using the maps $\alpha$ and $\rho$, and also show that the orthogonal of a free $R$-code is free.
\begin{proposition}\label{PropFreeAlphaPi}
The following are equivalent
\begin{enumerate}
\item $C$ is a free $R$-code,
\item $\alpha^{-1}(\soc(C))=\rho(C)$, and
\item $\fr(C)=\type(C)$.
\end{enumerate}

\end{proposition}
\begin{proof}
Suppose that $C$ is free. Since $C$ is an injective $R$-module, we have $R^n\cong C\oplus N$ for some $R$-submodule $N\subseteq R^n$. Then, every basis $\{e_1,\ldots e_t\}$ for $C$ can be extended to a basis $\{e_1,\ldots e_t, w_1,\ldots, w_{n-t}\}$ of $R^n$. Thus, $\soc(C)$ has $ze_1,\ldots,ze_t$ as a basis over $K$. Consequently, $\alpha^{-1}(\soc(C))=\rho(C)$.

Now assume that $\alpha^{-1}(\soc(C))=\rho(C)\cong \soc(C)$. By Lemma \ref{freerank}, $\fr(C)=\type(C)$. 

Finally, consider that $\fr(C)=\type(C)$. Let $t=\fr(C)=\type(C)$.
We have $C\cong R^{t}\oplus N$. Since 
$$t=\dim_K \soc(C)=\dim_K \soc(R^t)+\dim_K \soc(N)=t+\type(N),$$
it follows that $\type(N)=0$, and hence, $N=0$. Therefore, $C$ is free.
\end{proof}

The intersection between a code and its orthogonal, called the hull, is crucial for designing codes that can detect and correct errors efficiently, optimize information protection, and develop secure communication and cryptographic systems~\cite{Leon1982, Sendrier2000}. Recently, the hull has become essential in the development of quantum codes~\cite{Sarah2024}. We present some results for the hull of a code over a Frobenius ring.

\begin{definition}
The hull of an $R$-code $C$ is defined by $\hull(C)=C\cap C^{\perp}$. We say that a code is linear complementary dual, or LCD, if $\hull(C)=0$.
\end{definition}

\begin{lemma}\label{LemmaTypebound}
We have that $\lambda(C) \leq \type(C)\cdot \lambda(R).$ Even more,
$\lambda(C)=\type(C)\cdot \lambda(R)$ if and only if $R$ is free.
\end{lemma}
\begin{proof}
As $R$ is Frobenius, there exists an $R$-linear injection $\phi: C\hookrightarrow R^{\type(C)}$. Then, $\lambda(C)\leq \type(C)\cdot \lambda(R)$. We also have that  $\phi$ is an isomorphism, which is equivalent to $C$ being free, if and only if $\lambda(C)=\type(C)\cdot \lambda(R)$.
\end{proof}

We are now ready to show that LCD codes in Frobenius rings are free. This was previously proven for finite rings~\cite[Theorem 2]{LCD}. Our strategy is similar, with the main difference that we use the length instead of the cardinality. 
\begin{theorem}\label{ThmLCDfree}
If $R$ is Frobenius and $C$ is an LCD $R$-code, then $C$ is a free $R$-code.
\end{theorem}
\begin{proof}
We proceed by contradiction. Assume that $C$ is not free. By Theorem~\ref{Lemma_Free_iff_Free}, $C^\perp$ is also not free. %Therefore, by Lemma~\ref{LemmaTypebound}, we get that $\lambda(C)<\type(C) \cdot \lambda(R)$ and $\lambda(C^\perp)<\type(C^\perp) \cdot \lambda(R)$.
The short exact sequence
\begin{equation}\label{EqSES}
0\to \hull(C)\to C\oplus C^\perp \to  C+ C^\perp\to 0
\end{equation}
implies that $C\oplus C^\perp \cong  C+ C^\perp$. Therefore, $\type(C)+\type(C^\perp)=\type(C+ C^\perp)$. Thus,
\begin{align*}
\lambda(C+C^\perp)= \lambda(C)+\lambda(C^\perp)& <\type(C)\cdot \lambda(R)+\type(C^\perp)\cdot\lambda(R) \quad \text{(by Lemma~\ref{LemmaTypebound})}\\
&=(\type(C)+\type(C^\perp))\cdot\lambda(R)\\
&=n\cdot\lambda(R).
\end{align*}
We also have that
$$
\lambda(C+C^\perp)= \lambda(C)+\lambda(C^\perp)=\lambda(R^n)=n\cdot\lambda(R) \quad \text{(by Lemma~\ref{Lemma_Dimension_analogous})},
$$
which leads to a contradiction. Thus, $C$ must be free.
\end{proof}

%%%%%%%%%%%%%%%%%%%%%%%%%%%%%%%%%%
\section{Codes over Artinian local rings}\label{SecCodes}
In this section, $C$ denotes an $R$-code, where $(R,\m, K)$ is an Artinian local ring. For $\cA\subseteq [n]$, we set $C(\cA)=\{ v\in C\; |\; \supp(v)\subseteq \cA\}$ and $H_\cA=R^n(\cA)$. Observe that $H_\cA$ is the $R$-code generated by $\{e_i\; | \; i\in\cA\}$, where $\{e_1,\ldots, e_n\}$ is the canonical basis of $R^n$. 

\begin{remark}\label{RemDistMap}
If the cardinality of $\cA$ is $\hd(C)-1$, we have that $ R^n/H_\cA\cong R^{n+1-\hd(C)}$. In addition, as $H_\cA\cap C=0$, the restriction map $\phi:C\to  R^n/H_\cA$ is injective.
%Let $\pi:R^n\to R^n/H_\cA$ be the projection map.  We note that 
\end{remark}

If $R$ is finite and $C\subseteq R^n$, the Singleton bound~\cite{Singleton} states $|C|\leq |R|^{n+1-\hd(C)}$. By Remark~\ref{RemLenCardinality}, the Singleton bound is equivalent to $\lambda(C) \leq \lambda(R) \left(n + 1 - \hd(C) \right)$. We now show that the last bound still holds even for infinite rings.
%In other words,
%\begin{equation}\label{SingletonBoundOriginal}
%$\hd(C)\leq n-\frac{\lambda(C)}{\lambda(R)}+1$.
%\end{equation}
%We now show that Equation \ref{SingletonBoundOriginal} still holds if the rings is not finite. Furthermore, we show a related inequality.

\begin{theorem}\label{ThmSing}
If $R$ is a local Artinian ring, we have that
\[
\hd(C)\leq n-\frac{\lambda(C)}{\lambda(R)}+1
\qquad
\text{ and }
\qquad
\hd(C)\leq n-\frac{\type(C)}{\type(R)}+1.
\]
\end{theorem}
\begin{proof}
Take $\cA\subseteq [n]$ of cardinality $\hd(C)-1$. By Remark~\ref{RemDistMap}, we have that
\begin{align*}
\lambda(C)\leq \lambda( R^n/H_\cA)=\lambda( R^{n+1-\hd(C)})&=(n+1-\hd(C))\lambda(R) \quad \text{ and }\\
\type(C)\leq \type( R^n/H_\cA)=\type( R^{n+1-\hd(C)})&=(n+1-\hd(C))\type(R).
\end{align*}
These inequalities give the desired results.
\end{proof}
\begin{remark}\label{Rmk.tl}
From Theorem~\ref{ThmSing}, it is natural to ask if there is a relation between the quotients $\lambda(C)/\lambda(R)$ and $\type(C)/\type(R)$. The answer is that one or the other can be larger, as the following explanation shows.

In a chain ring, we have that  
$$\type(C)=\frac{\type(C)}{\type(R)}\leq \frac{\lambda(C)}{\lambda(R)}.$$
However, this is not the case in other rings. Suppose that
$S=K[x,y_1,\ldots,y_t]$, $\m=(x,y_1,\ldots,y_t)$, $R=S/(\m^3,xy_j,y_i y_j\; |\; i,j=1,\ldots, t)$, $C=R x$, and $t\geq 3$. Then,
$\soc(R)=Kx^2+\ldots+Ky_t$, $\type(R)=t+1$, $\lambda(R)=t+3$, $\soc(R)=Kx^2$, $\type(C)=1$, $\lambda(C)=2$, 
$\frac{\lambda(C)}{\lambda(R)}=\frac{2}{t+3}$, and $\frac{\type(C)}{\type(R)}=\frac{1}{t+1}.$
\end{remark}

As a consequence of Theorem~\ref{ThmSing}, we recover a bound for the minimum distance of an $R$-code $C$ over a chain ring given by Samei and Mahmoudi in terms of $\mu(C)$, the minimum number of generators of $C$.
\begin{corollary}[{\cite[Theorem 3.7(2)]{Samei_Mahmoud}}]
Let $R$ be a chain ring. For an $R$-code $C\subseteq R^n$, we have that
\[
\hd(C)\leq n-\mu(C)+1.
\]
\end{corollary}
\begin{proof}
This is a consequence of Theorem~\ref{ThmSing} because the minimum number of generators and the type of an $R$-code are equal when $R$ is a chain ring.
\end{proof}
We also recover a bound for the minimum distance of an $R$-code $C$ over a finite Frobenius ring given by Dougherty in terms of linear injections.
\begin{corollary}[{\cite[Theorem 4.12]{BookCodesRings}}]
If $R$ is a finite Frobenius ring, then
\[
\hd(C)\leq n+1-\min\{\theta\; |\; \exists R-\hbox{linear injection } C\hookrightarrow R^\theta\}.
\]
\end{corollary}
\begin{proof}
This is a consequence of Theorem~\ref{ThmSing} because $$\type(C)=\min\{\theta\; |\; \exists R-\hbox{linear injection } C\hookrightarrow R^\theta\}$$ for an $R$-code $C$ when $R$ is a finite Frobenius ring.
\end{proof}

Motivated by Theorem~\ref{ThmSing}, we introduce codes classified by the bound on their minimum distance.
\begin{definition}
The inequalities presented in Theorem~\ref{ThmSing} are called MDS (maximum distance separable) and MDT (maximum distance type) bounds, respectively (from left to right). A code achieving the MDS (MDT) bound is called an MDS (MDT) code.
\end{definition}

MDT codes over a chain ring coincide with MDR codes introduced in the previous work~\cite{MDR_MDS}.

Despite the similarity between the two inequalities in Theorem~\ref{ThmSing}, the two definitions do not coincide in general, as the following example shows.
\begin{example}\label{ex.mdtnotmdr} Let $K$ be any field and define
$R=K[x]/(x^t)$, with $t\geq 2$. Define the $R$-code $C=(x^{t-1},x^{t-1})\in R^2$. We note that $d_H(C)=2$ and $\type(C)=1$. Then, $C$ is an MDT code, but not an MDS code.
\end{example}

The previous example says that MDT does not imply MDS. The following result shows that MDS implies MDT. It is interesting to note, however, that the implication MDS $\rightarrow$ MDT means nothing about the relation between the quotient of the lengths and the quotient of the types; see Remark~\ref{Rmk.tl} for more details.

\begin{theorem}\label{ThmEquivMDS}
If $R$ is a local Artinian ring, then $C$ is MDS if and only if $C$ is a free MDT.
\end{theorem}
\begin{proof}
By Remark \ref{RemDistMap}, we have an injective map $\phi: C\to R^{n-\hd(C)+1}$.

($\Rightarrow$) We first assume that $C$ is an MDS code. Then, $\lambda(C)=(n-\hd(C)+1)\lambda(R)$ and  $\phi: C\to R^{n-\hd(C)+1}$ is an isomorphism, which implies that $C$ is a free $R$-code. In addition, $\type(C)=(n-\hd(C)+1)\type(R)$, thus, $C$ is an MDT $R$-code.

($\Leftarrow$) If $C$ is a free and an MDT code, then $\type(R)\lambda(C)=\lambda(R)\type(C)$ and $\type(C)=(n-\hd(C)+1)\type(R)$. Combining the last two equalities, we obtain
$$
\frac{\lambda(C)}{\lambda(R)}=\frac{\type(C)}{\type(R)}=n-\hd(C)+1.
$$
Hence, $C$ is an MDS code.
\end{proof}

\begin{definition}\rm
Take $\mathcal{A}\subseteq[n]$, and let $\pi_{\mathcal{A}}:R^n\rightarrow R^{|\mathcal{A}|}$ be the projection that drops the coordinates outside $\mathcal{A}$. The puncturing and shortening of $C\subseteq R^n$ over $\mathcal{A}$ are defined respectively by
\[
C|_{\mathcal{A}}=\pi_{\mathcal{A}^c}(C) \qquad \text{ and } \qquad
C_{\mathcal{A}}=\pi_{\mathcal{A}^c}(C\cap{H}_\mathcal{A}^c).
\]
We refer to them as punctured and shortened codes, respectively.
\end{definition}

Puncturing and shortening are common operations over codes in coding theory that reveal structural properties of a code. For example, when $R$ is a field, the MDS property is equivalent to certain properties of the shortened and punctured code. The following proposition demonstrates that some of these properties remain valid for Artinian rings.

\begin{proposition}\label{Prop.MDSupp}
Let $C\subseteq R^n$ be a code of minimum distance $d$.
\begin{enumerate}    
\item $C$ is MDS if and only if for any $\mathcal{A}\subseteq [n]$ of cardinality $d-1$, $C|_{\mathcal{A}}=R^{n-d+1}$.

\item Assume that $R$ is Frobenius. $C$ is MDS if and only if for any $\mathcal{A}\subseteq [n]$ of cardinality $n-d$, $C_\mathcal{A}\neq 0$.
\end{enumerate}
\end{proposition}

\begin{proof}
(1) This is a consequence of the proof of Theorem~\ref{ThmSing}. (2) Assume that there is $\mathcal{A}\subseteq[n]$ with cardinality $n-d$ such that $C_{\mathcal{A}}=0$. This implies that $C\cap H_{A^c}=0$. By a similar argument as in Remark~\ref{RemDistMap}, we obtain that $\lambda(C)\leq (n-d)\lambda(R)$. Therefore, $C$ cannot be an MDS code.

Assume that for any $\mathcal{A}\subseteq [n]$ of cardinality $n-d$, we have $\mathcal{A}\cap C\neq 0$. For $1\leq i\leq n-d+1$, define $\mathcal{A}_i=\{i,i+1,\ldots,d+i-1\}$. Then $\lambda(C\cap\mathcal{A}_i)\geq 1$, and as $\mathcal{A}_i\cap\mathcal{A}_j\cap C=\{0\}$, we obtain
$$\type(C)\geq \sum_{i=1}^{n-d+1}\type(C\cap \mathcal{A}_i)\geq n-d+1.$$
By Theorem~\ref{ThmSing}, we obtain the result.
\end{proof}

If $R$ is not Frobenius, Proposition~\ref{Prop.MDSupp}~(2) may not be true. For example, take $R=\mathbb{F}_2[x,y]/(x^2,xy,y^2)$ and let $C$ be the $R$-code generated by $\{(x,x,0),(0,x,x)\}$. In this case, the code \(C\) has length \(n = 3\) and minimum distance \(d = 2\). For any subset \(\mathcal{A} \subseteq [3]\) with \(|\mathcal{A}| = n - d = 1\), the equality \(C|_{\mathcal{A}} = R^{2}\) does not hold. Indeed, the punctured codes \(C|_{\mathcal{A}}\) form proper subsets of \(R^{2}\) with significantly fewer elements.

A desired result in coding theory is that an optimal code (with respect to a certain Singleton-like bound) also has an optimal dual. When the code is defined over a finite field, this dual property is true in the case of the Hamming, the rank, and, sometimes, the sum-rank metric. We now verify that the orthogonal code of an MDS code is also an MDS code when $R$ is an Artinian local ring.

\begin{lemma}\label{LemmmaDirectSummandMDS}
Let $C\subseteq R^n$ be an MDS code. There exists a surjective map $\eta: R^n\to C$ of $R$-modules such that $\eta_{|_C}=1_C$. In particular, $C$ is a direct summand of $R^n$.
\end{lemma}
\begin{proof}
Let $\cA\subseteq [n]$ be a subset of cardinality $\hd(C)-1$. Recall $H_\cA=R^n(\cA)$ and let $\pi:R^n\to R^n/H_\cA$ be the projection map. We have that $ R^n/H_\cA\cong R^{n+1-\hd(C)}$ by Remark~\ref{RemDistMap}, and the restriction $\phi:C\to  R^n/H_\cA$ is an isomorphism because $C$ is MDS. The result follows from taking $\eta=\phi^{-1}\circ\pi$.
\end{proof}

\begin{lemma}\label{LemmaRankOrthogonal}
If $C$ is MDS, then $C^\perp$ is free,  $\rank(C) + \rank(C^\perp)=n$, and $(C^\perp)^\perp=C$.
\end{lemma}
\begin{proof}
Let $W$ be the canonical module of $R$ and $S=R \rtimes W$ the Nagata idealization of $W$, which in this case equal to the injective hull of $R$, $E_R(K)$. That is, $S\cong R\oplus W$ and $S$ is an $R$-module. The product is defined by $(a_1+w_1)(a_2+w_2)=a_1 a_2+a_2 w_1+a_1 w_2$ where $a_i\in R$ and $w_i\in W$. Since \(W\) is the canonical module of \(R\), which is injective, S is an artinian and injective ring, hence \(S\) is Frobenius.

Observe that $R\subseteq S$, and so $R^n\subseteq S^n$. Let $C_S$ be the $S$-code generated by $C$, which is isomorphic to $C\otimes S\to C_S$ via $v\otimes s\mapsto sv$. We have that $C\otimes_R S$ is a direct summand of $S^n$ because $C$ is a direct summand of $R^n$ by Lemma~\ref{LemmmaDirectSummandMDS}. Then, $C_S$ is a free $S$-code and $\rank_R(C)=\rank_S(C_S)$. Therefore, we obtain that
$C_S^\perp$ is a free code of $$ \rank_S(C^\perp)=\frac{\lambda_S(C_S^\perp)}{\lambda_S(S)}=n-\frac{\lambda_S(C_S)}{\lambda_S(S)}=n-\rank_S(C)$$
by Lemma~\ref{Lemma_Dimension_analogous} and Theorem~\ref{Lemma_Free_iff_Free}.

Observe $C^\perp\subseteq C_S^\perp$. Moreover, if $\psi :S^n\to R^n$ is the natural projection, we have $\psi(C^\perp_S)=C^\perp$. As the kernel of $\psi$ is $(0\oplus W)S^n$, it follows that $\psi(C^\perp\oplus 0)=\psi((C_S)^\perp)$. Since $C^\perp_S$ is a free direct summand of $S^n$, we can find a basis of $C^\perp_S$ in $C^\perp$ using  that $\psi(C^\perp\oplus 0)=\psi((C_S)^\perp)$ and Nakayama's Lemma. Hence, $C^\perp$ is a free $R$-code because $C^\perp_S$ is a free $S$-code generated by elements in $C$. Moreover, $\rank_S(C_S^\perp)=\rank_R(C^\perp)$.

We now show the last claim. We note that $C\subseteq (C^\perp)^\perp$. Since $C$ and $(C^\perp)^\perp$ are both free of the same rank, they have the same length, and the claim follows.
\end{proof}

\begin{theorem}\label{Theo.MDSduals}
If $R$ is a local Artinian ring, $C$ is MDS if and only if $C^\perp$ is MDS.
\end{theorem}
\begin{proof}
Assume that $C$ is an MDS code of minimum distance $d$. For any $\mathcal{A}\subseteq[n]$ with cardinality $|\cA|\geq d-1$, we have $\lambda(C(\cA))=\lambda(R)(|\cA|-d+1)$ and $C+H_{\cA}=R^n$. Taking orthogonals, we obtain $C^\perp\cap H_{\cA^c}=\{0\}$, then the minimum distance of $C^\perp$ is at least $n-d+2$. We have that $C^\perp$ is free of rank $n-\rank(C)$ by Lemma~\ref{LemmaRankOrthogonal}. Thus,
\begin{align*}
\hd(C^\perp)
&\leq n-\frac{\lambda(C^\perp)}{\lambda(R)}+1 \text{ (by Theorem~\ref{ThmSing}) }\\
& =n-\rank(C^\perp)+1 \text{ (because $C^\perp$ is free) }\\
& =n-(n-\rank(C))+1\\
&=\rank(C)+1\\
&= \frac{\lambda(C)}{\lambda(R)}+1\text{ (because $C^\perp$ is free by Theorem~\ref{ThmEquivMDS}) }\\
&=n-d+2.
\end{align*}
We conclude that $C^\perp$ is MDS of minimum distance $n-d+2$.
\end{proof}

%%%%%%%%%%%%%%%%%%%%%%%%%%%%%%%%%%
\section{Enumerator polynomials and MacWilliams identities}\label{SecMacWilliams}
In this section, $C$ denotes an $R$-code, where $(R,\m, K)$ is an Artinian local ring. We define an analogous polynomial to the weight enumerator, but it depends on the length of the code rather than its cardinality. We then obtain a MacWilliams identity theorem for Frobenius local rings, even if they are not finite. To define the new polynomial, 
we introduce a variable $z$ which, when $K = \mathbb{F}_q$ and the polynomial is evaluated at $z = q$, yields the classical MacWilliams identity.

%To find the new polynomial, we introduce a new variable $z$, which, when evaluated at $q$, yields the classical MacWilliams identity.

\begin{definition}
If $R$ is finite, the weight enumerator of $C$ is defined by
$$W_C(x,y)=\sum^n_{i=1} A^C_i x^{n-i}y^{i},$$
where $A^C_i=|\{v\in C\ :\ wt(v)=i\}|$. 
\end{definition}
\begin{definition}
We define  the length enumerator of $C$ by 
\[
L_C(x,y,z) = \sum_{\substack{\cA \subseteq [n] }} 
z^{\lambda(C(\cA))} x^{n-|\cA|} y^{|\cA|}.
\]
%$$L_C(x,y,z)=\sum^n_{i=1} z^{\lambda(C(\cA))}x^{n-i}y^{i}.$$
\end{definition}

Unlike the weight enumerator, the length enumerator does not count individual codewords based on their weight. Instead, it groups codewords according to subsets of their support and assigns weights accordingly. The following result shows the relation between the weight and length enumerators. Specifically, we prove that if $|K| = q$, the evaluation of $L_C(x,y,z)$ at $z=q$ recovers the weight enumerator.

\begin{proposition}\label{Prop L vs W finite}
If $R$ is finite and $K=\mathbb{F}_q$, then
$L_C(x,y,q)=W_C(x+y,y)$.
\end{proposition}
\begin{proof}
By Remark \ref{RemLenCardinality}, we know $q^{\lambda(C(\cA))}=|C(\cA)|$ for any $\cA\subseteq [n]$. Then 

\[
L_C(x,y,q) =\sum_{\substack{\cA \subseteq [n] }} q^{\lambda(C(\cA))} x^{n-|\cA|} y^{|\cA|}
=  \sum_{\substack{\cA \subseteq [n]}} |C(\cA)| x^{n-|\cA|} y^{|\cA|}.
\]

We group the vectors in $C(\cA)$ by their exact support. 
For a subset $\cB \subseteq \cA$, let $C_\cB = \{ v \in C : \supp(v) = \cB \}$, and write $|C(\cA)| = \sum_{\cB \subseteq \cA} |C_\cB|.$

Thus 
\[
L_C(x,y,q) = \sum_{\substack{\cA \subseteq [n] }} \sum_{\cB \subseteq \cA} |C_\cB| x^{n-|\cA|} y^{|\cA|}.
\]

Now, fix $\cB \subseteq [n]$ of size $|\cB| = i$. 
For each $\cB$, the number of subsets $\cA$ such that $\cB \subseteq \cA$ and $|\cA| = i+j$ is $\binom{n-i}{j}$. 
Therefore, 

\begin{align*}
L_C(x,y,q) 
  &= \sum_{\substack{\cB \subseteq [n] \\ |\cB| = i}} |C_\cB| 
     \sum_{j=0}^{n-i} \binom{n-i}{j} x^{\,n-i-j} y^{\,i+j} \\[6pt]
  &= \sum_{i=0}^{n} \sum_{\substack{\cB \subseteq [n] \\ |\cB| = i}} 
     |C_\cB| \left( \sum_{j=0}^{n-i} \binom{n-i}{j} x^{\,n-i-j} y^{\,j} \right) y^i \\[6pt]
  &= \sum_{i=0}^{n} \sum_{\substack{\cB \subseteq [n] \\ |\cB| = i}} 
     |C_\cB| (x+y)^{\,n-i} y^i \\[6pt]
  &= \sum_{i=0}^{n} 
     \left( \sum_{\substack{\cB \subseteq [n] \\ |\cB| = i}} |C_\cB| \right) (x+y)^{\,n-i} y^i \\[6pt]
  &= \sum_{i=0}^{n} A^C_i (x+y)^{\,n-i} y^i \\[6pt]
  &= W_C(x+y,y),
\end{align*}

which completes the proof.

%

%\begin{align*}

%L_C(x,y,q)&=\sum_{\cA\subseteq[n]} |C(\cA)| x^{n-i}y^{i}\\
%&=\sum_{\cA\subseteq [n]}\sum_{\cB\subseteq \cA}|\{v\in C\ :\ \supp(v)=\cB\}| x^{n-i}y^{i}.
%\end{align*}
%
%If $\cB$ is fixed of size $i$, there are ${n-i \choose j}$ subsets $\cA$ such that $\cB\subseteq \cA$ and $|\cA|=i+j$. Thus,
%
%\begin{align*}
%L_C(x,y,q)&=\sum_{i=0}^n \sum_{\substack{\cB\subseteq [n] \\ |\cB|=i}}\sum_{j=0}^{n-i}{n-i\choose j}|\{v\in C : \supp(v)=\cB\}| x^{n-i-j}y^{i+j}\\        &=\sum_{i=0}^n\sum_{\substack{\cB\subseteq [n] \\ |\cB|=i}}\left(|\{v\in C : \supp(v)=\cB\}|\sum_{j=0}^{n-i}\left({n-i\choose j} x^{n-i-j}y^j\right) y^i\right)\\
%&=\sum_{i=0}^n\sum_{\substack{\cB\subseteq [n] \\ |\cB|=i}}\left(|\{v\in C : \supp(v)=\cB\}|\left(x+y\right)^{n-i} y^i\right)\\
%&=\sum_{i=0}^n A^C_i\left(x+y\right)^{n-i} y^i\\
%&=W_C(x+y,y),
%\end{align*}
%which completes the proof.
\end{proof}

Shiromoto studied the enumerator of linear codes and proved  Proposition~\ref{Prop L vs W finite} for finite fields~\cite{Shiromoto, Shiromoto1999}. This proposition extends that approach to a broader algebraic context.

\begin{definition}
We define the  $i$-th weight polynomial $g^C_{i}(z)$ by
$$
g^C_{i}(z)=\sum^i_{a=0} \sum_{|\cA|=a}(-1)^{i-a}\binom{n-a}{i-a}z^{\lambda(C(\cA))}.$$
\end{definition}

Regrouping the terms in $L_C(x,y,z)$ according to the size of the sets $\cA$, the $i$-th weight polynomial recovers the number of codewords of weight $i$.
\begin{proposition}\label{Prop G number codewords}
If $R$ is finite and $K=\mathbb{F}_q$, then
$g^C_i(q)=A^C_i$.
\end{proposition}
\begin{proof}
We have that
\begin{align*}
W_C(x,y)&=L_C(x-y,y,q) \text{ (by Proposition~\ref{Prop L vs W finite}) }\\
   &=\sum_{\cA\subseteq[n]} q^{\lambda(C(\cA))} (x-y)^{n-|\cA|}y^{|\cA|}\\
    &=\sum_{a=0}^n\sum_{\substack{\cA\subseteq [n]\\ |\cA|=a}}q^{\lambda(C(\cA))}(x-y)^{n-a}y^a\\
    &=\sum_{a=0}^n\sum_{\substack{\cA\subseteq [n]\\ |\cA|=a}}q^{\lambda(C(\cA))}\left(\sum_{i=0}^{n-a}(-1)^i{n-a\choose i} x^{n-a-i}y^i\right)y^a\\
    &=\sum_{a=0}^n\sum_{j=a}^n\sum_{\substack{A\subseteq [n]\\ |\cA|=a}}q^{\lambda(C(\cA))}(-1)^{j-a}{n-a\choose j-a}x^{n-j}y^j\\
&    =\sum_{j=0}^n\sum_{a=0}^j\sum_{\substack{\cA\subseteq [n]\\ |\cA|=a}}q^{\lambda(C(\cA))}(-1)^{j-a}{n-a\choose j-a}x^{n-j}y^j\text{ (by rearranging the sums) }\\
&=\sum_{j=0}^n g^C_i(q)x^{n-j}y^j.
\end{align*}
By the definition of the weight enumerator, we obtain $g^C_i(q)=A^C_i$.
\end{proof}

Proposition~\ref{Prop G number codewords} extends the proof for finite fields given by Shiromoto~\cite{Shiromoto}. 

%Proposition~\ref{Prop G number codewords} was proved for finite fields by Shiromoto~\cite{Shiromoto}. \hiram{say a bit more?}

\begin{corollary}\label{CorExtension}
Let $(S,\n,L)$ be a flat $R$-algebra such that $S$ is zero-dimensional ring with $\theta=\lambda_S(S/\m S)$.
Then, $g^{C\otimes_R S}_i(z)=g^C_i(z^\theta)$, where $C\otimes_R S\subseteq S^n$ is an $S$-code.
\end{corollary}
\begin{proof}
Take $\cA\subseteq [n]$ and recall that $H_\cA=R^n(\cA)$. Then, $(C\cap H_\cA)\otimes_R S=C\otimes_R S\cap H_\cA\otimes_R S$ because $S$ is flat. As $K\otimes_R S=S/\m S$, we get
$$
\lambda_S(C\otimes_R S(\cA))=\lambda_S (C\otimes_R S\cap H_\cA\otimes_R S)=\lambda_S((C\cap H_\cA)\otimes_R S)=\theta\lambda_R(C\cap H_\cA)
=\lambda_R(C(\cA)).
$$
Therefore, 
\begin{align*}
g^{C\otimes_R S}_{i}(z) & =\sum^i_{a=0} \sum_{|\cA|=a}(-1)^{i-a}\binom{n-a}{i-a}z^{\lambda_S(C\otimes_R S(\cA))}\\
&=\sum^i_{a=0} \sum_{|\cA|=a}(-1)^{i-a}\binom{n-a}{i-a}z^{\theta \lambda_R(C(\cA))}\\
&=\sum^i_{a=0} \sum_{|\cA|=a}(-1)^{i-a}\binom{n-a}{i-a}\left( z^\theta\right)^{\lambda_R(C(\cA))}
=g^{C}_{i}(z^\theta),
\end{align*}
which completes the proof.
\end{proof}

When $R$ is a finite field, Corollary~\ref{CorExtension} has a critical interpretation. Specifically, $\theta=1$ in this case, and the polynomial $g^C_i(z)$ determines the number of codewords that a code has when it is considered in a finite field extension.

\begin{example}\label{ExFieldExt}
Take $R=\FF_q$ and let $C\subseteq R^3$ be the code generated by the elements
$\begin{pmatrix} 1, 1, 0 \end{pmatrix}$ and $\begin{pmatrix} 1, 0, 1 \end{pmatrix}$. Then, 
$$g^C_0(z)=1,\;
g^C_1(z)=0,\;
g^C_2(z)=3z-3\;
\text{ and }\;
g^C_3(z)=z^2-3z+2.$$
We have that $C$ is a constant-weight code if and only if $q=2$. We can also see that the growth of the number of codewords with weight $3$ is quadratic with respect to $q$.
\end{example}

The proposition \ref{Prop L vs W finite} motivates the definition of weight enumerators for rings that are not necessarily finite.

\begin{definition}
We define the weight enumerator of $C$ by 
$$W_C(x,y,z)=\sum^n_{i=1} g^C_{i}(z)x^{n-i}y^{i}.$$
\end{definition}

\begin{proposition}\label{Prop L vs W infinite}
We have that $L_C(x,y,z)=W_C(x+y,y,z)$.
\end{proposition}
\begin{proof}
This is analogous to Proposition~\ref{Prop L vs W finite} and the definition of $g^C_{i}(z)$.
\end{proof}

As an example, we compute the weight enumerator of MDS codes. For $\cA\subseteq [n]$, recall that $C(\cA)=\{ v\in C\; |\; \supp(v)\subseteq \cA\}$.
\begin{proposition}
$C$ is MDS iff $\lambda(C(\cA))=\max\{0,|\cA|-d+1\}\lambda(R)$ for any $\cA\subseteq [n]$.
\end{proposition}
\begin{proof}
($\Rightarrow$) Assume $C$ is MDS. Let $\cA\subseteq [n]$ be a subset of cardinality $a = |\cA|$. If $a\leq d-1$, then $C(\cA)=\{0\}$. If $a=d-1$, we obtain that $C+H_{\cA}=R^n$ by Remark~\ref{RemDistMap}. If $a>d-1$, then $C+H_{\cA}=R^n$, and so,
$$n\lambda(R)=\lambda(C)+\lambda(H_{\cA})-\lambda(C(\cA))\Longrightarrow \lambda(C(\cA))=\lambda(R)(n-d+1+a-n)=\lambda(R)(a-d+1).$$
In any case of $a$, the result follows.

($\Leftarrow$) By taking $\cA=[n]$, we obtain $C(\cA)=C$ and $\lambda(C(\cA))=(n-d+1)\lambda(R)$. 
\end{proof}

\begin{corollary}
If $C$ is MDS, then
$$g^C_i(z)=\sum_{a=0}^i (-1)^{i-a} {n\choose i}{i\choose a}z^{\max\{0,a-d+1\}\lambda(R)}.$$
\end{corollary}

\begin{proof}
As $\lambda(C(\mathcal{A}))$ only depends on the cardinality of $\mathcal{A}$, then
\begin{align*}
g^C_i(z)&=\sum_{a=0}^i\sum_{|\mathcal{A}|=a}(-1)^{i-a}{n-a\choose i-a}z^{\lambda(C(\mathcal{A}))}=\sum_{a=0}^i(-1)^{i-a}{n\choose a}{n-a\choose i-a}z^{\max\{a-d+1,0\}}\\
&=\sum_{a=0}^i (-1)^{i-a} {n\choose i}{i\choose a}z^{\max\{0,a-d+1\}\lambda(R)}.
\end{align*}
This completes the proof.
\end{proof}
\begin{remark}
Observe that $\sum_{a=0}^i (-1)^{i-a}{i\choose a}=0$. So, if $C$ is MDS, then
\begin{align*}
g^C_i(z)-0&={n\choose i}\left(\sum_{a=0}^i (-1)^{i-a}{i\choose a}\left(z^{\max\{0,a-d+1\}\lambda(R)}-1\right)\right)\\
&={n\choose i}\left(\sum_{a=d}^i (-1)^{i-a}{i\choose a}\left(z^{(a-d+1)\lambda(R)}-1\right)\right).
\end{align*}
Thus, when $R=\mathbb{F}_q$ and $z=q$, we recover the weight distribution of an MDS code.
\end{remark}

Since the dual of an MDS code is MDS, we can say that the weight enumerator polynomials of an MDS code and its dual are determined by each other. When $R$ is a field or a finite Frobenius ring, this is well-known and always true. In the general case of Artinian rings, the duality of enumerator polynomials is also tied to the property of being Frobenius.

\begin{lemma}\label{LemmaLenA}
If $R$ is a Frobenius ring, then
$$
\lambda(C^\perp(\cA))=\lambda(C^\perp)+\lambda(C(\cA^c))-\lambda(R)|\cA^c|.
$$
\end{lemma}
\begin{proof}
Note
$(C^\perp (\cA))^\perp = (C^\perp\cap H_{\cA})^\perp
=(C^\perp)^\perp+(H_{\cA})^\perp=C+H_{\cA^c}$.
Thus,
\begin{align*}
\lambda((C^\perp (\cA))^\perp)& =\lambda(C+H_{\cA^c})\\
& =\lambda(C)+\lambda(H_{\cA^c})-\lambda(C\cap H_{\cA^c}) \nonumber\\
& =n\lambda(R)-\lambda(C^\perp)+\lambda(H_{\cA^c})-\lambda(C(\cA^c)).\label{EqCcA1}
\end{align*}
As $\lambda((C^\perp (\cA))^\perp) = n\lambda(R) - \lambda(C^\perp (\cA))$ by Lemma~\ref{Lemma_Dimension_analogous}, we obtain the result.
\end{proof}

\begin{proposition}\label{prop.mwi}
If $R$ is a Frobenius ring, then
$$L_{C^\perp}(x,y,z)=\frac{1}{z^{\lambda(C)}}L_C(z^{\lambda(R)}y,x,z).$$
\end{proposition}

\begin{proof}
We have that
\begin{align*}
L_{C^\perp}(x,y,z)&=\sum_{\cA\subseteq [n]} z^{\lambda(C^\perp(\cA))}x^{n-|\cA|}y^{|\cA|}\\
&=\sum_{\cA\subseteq [n]}z^{\lambda(C^\perp)+\lambda(C(\cA^c))-\lambda(R)|\cA^c|}x^{n-|\cA|}y^{|\cA|}\text{ (by Lemma~\ref{LemmaLenA}) }\\
&=\sum_{\cA\subseteq [n]}z^{\lambda(C^\perp)+\lambda(C(\cA))-\lambda(R)|\cA|}x^{|\cA|}y^{n-|\cA|}\text{ (by taking complements) }\\
&=\sum_{\cA\subseteq [n]}z^{n\lambda(R)-\lambda(C)+\lambda(C(\cA))-\lambda(R)|\cA|}x^{|\cA|}y^{n-|\cA|}\\
&=\frac{1}{z^{\lambda(C)}}\sum_{\cA\subseteq [n]}z^{n\lambda(R)+\lambda(C(\cA))-\lambda(R)|\cA|}x^{|\cA|}y^{n-|\cA|}\\
&=\frac{1}{z^{\lambda(C)}}\sum_{\cA\subseteq [n]}z^{\lambda(C(\cA))}x^{|\cA|}(z^{\lambda(R)}y)^{n-|\cA|}\\
&=\frac{1}{z^{\lambda(C)}}L_C(z^{\lambda(R)}y,x,z).
\end{align*}
Thus, we obtain the result.
\end{proof}

\begin{theorem}[MacWilliams Identity]\label{Theo.MWid}
If $R$ is Frobenius, then
$$W_{C^\perp}(x,y,z)=\frac{1}{z^{\lambda(C)}}W_C(x+zy-y,x-y,z).$$
\end{theorem}
\begin{proof}
We have that
\begin{align*}
W_{C^\perp}(x,y,z)&=L_{C^\perp}(x-y,y,z) \text{ (by Proposition~\ref{Prop L vs W infinite}) }\\
&=\frac{1}{z^{\lambda(C)}}L_{C}(z^{\lambda(R)}y,x-y,z)\text{ (by Proposition~\ref{prop.mwi}) }\\
&=\frac{1}{z^{\lambda(C)}}W_C(x-y+z^{\lambda(R)}y,x-y,z)\text{ (by Proposition~\ref{Prop L vs W infinite}) }.
\end{align*}
Thus, we obtain the result.
\end{proof}

If $R$ is a field of size $q$, the evaluation of the weight enumerator polynomial at $z=q$ recovers the classical MacWilliams identities. Furthermore, similar to the finite setting, the following result shows that if an $R$-code satisfies the MacWilliams identities, then the ring must be Frobenius.

\begin{corollary}\label{25.06.13}
A ring $R$ is Frobenius if and only if 
$$W_{C^\perp}(x,y,z)=\frac{1}{z^{\lambda(C)}}W_C(x+zy-y,x-y,z)$$
for any $n\in\mathbb{N}$ and any $R$-code $C\subseteq R^n$.
\end{corollary}

\begin{proof}
For simplicity, we prove the equivalent statement: $R$ is Frobenius if and only if for any code $C\subseteq R^n$,
$$L_{C^\perp}(x,y,z)=\frac{1}{z^{\lambda(C)}}L_C(z^{\lambda(R)}y,x,z).$$

The necessity is given by Proposition \ref{prop.mwi}. Sufficiency is given by the fact that
$$z^{\lambda(C^\perp)}=L_{C^\perp}(0,1,z)=\frac{1}{z^{\lambda(C)}}L_C(z^{\lambda(R)},0,z)=\frac{z^{n\lambda(R)}}{z^{\lambda(C)}}$$
and by Lemma~\ref{Lemma_Dimension_analogous}.
\end{proof}

\bibliographystyle{alpha}
\bibliography{References}

\end{document}